\documentclass[12pt,reqno]{amsart}
\usepackage{amssymb}
\usepackage{longtable}
\theoremstyle{plain}
\newtheorem{theorem}{Theorem}
\newtheorem{lemma}{Lemma}

\newtheorem{proposition}{Proposition}[section]
\theoremstyle{proof}
\theoremstyle{definition}

\newtheorem{definition}{Definition}[section]

\theoremstyle{remark}
\newtheorem{remark}{Remark}
\theoremstyle{lamma}

\newcommand{\Z}{\mathbb{Z}}

\newcommand{\C}{\mathbb{C}}

\newcommand{\mf}{\mathfrak}

\numberwithin{equation}{section}
\numberwithin{lemma}{section}
\numberwithin{theorem}{section}

\usepackage{amssymb, amsmath, amsthm}
\usepackage[breaklinks]{hyperref}
\theoremstyle{thmrm}

\usepackage{graphicx}

\begin{document} 
\title[]{Simple modules for twisted Hamiltonian extended affine Lie
algebras}
\author[]{Santanu Tantubay}
\address{Santanu Tantubay:  The Institute of Mathematical Science, A CI of Homi Bhaba National Institute,IV Cross Road, CIT Campus
Taramani
Chennai 600 113
Tamil Nadu, India.}
\email{santanutant@imsc.res.in, tantubaysantanu@gmail.com}

\author[]{ Priyanshu Chakraborty}
\address{Priyanshu Chakraborty:School of Mathematical Sciences, Ministry of Education Key Laboratory of Mathematics and Engineering Applications and Shanghai Key Laboratory of PMMP,
		East China Normal University, No. 500 Dongchuan Rd., Shanghai 200241, China.}
\email{ priyanshu@math.ecnu.edu.cn, priyanshuc437@gmail.com}
\author[]{Punita Batra}
\address{Punita Batra: Harish-Chandra Research Institute, A CI of Homi Bhaba National Institute,
Chhatnag Road, Jhunsi, Allahabad 211 019, India,}
\email{batra@hri.res.in}


\maketitle
\begin{abstract} 
In this paper we consider the twisted Hamiltonian extended affine Lie algebra (THEALA). We classify the irreducible integrable modules for these Lie algebras with finite dimensional weight spaces when the finite dimensional center acts non-trivially.  This Lie algebra has a triangular decomposition, which is different from the natural triangular decomposition of twisted full toroidal Lie algebra. Any irreducible integrable module of it is a highest weight module with respect to the given triangular decomposition. In this paper we describe the highest weight space in detail.

\end{abstract}

\section{Introduction} 
Extended affine Lie algebras (EALAs) form a category of important Lie algebras consisting of finite dimensional simple Lie algebras, affine Lie algebras and class of some other Lie algebras. Twisted toroidal extended affine Lie algebras are examples of EALAs. The structure theory of EALAs have been developed by several mathematicians like Allison, Azam, Berman, Gao, Neher, Pianzola and Yoshii (see \cite{[2]}, \cite{[16]}, \cite{[17]} and references therein). In 2004, Rao classified all the irreducible integrable modules for toroidal Lie algebras with finite dimensional weight spaces. In 2005, Rao and Jiang considered the full toroidal Lie algebra and they identified the highest weight space with Jet modules of derivations of Laurent polynomial rings. Later in 2018, Rao and Batra considered the more general Lie algebra, called twisted full toroidal Lie algebra and they classified all the irreducible integrable weight modules for this Lie algebra. Full toroidal or twisted full toroidal Lie algerbas are not examples of extended affine Lie algebras. So instead of adding full derivations one adds a subalgebra of derivations space in order to make it an extended affine Lie algebra. Twisted toroidal extended affine algebra is such an example. In 2020, Rao, Batra, Sharma classified all the irreducible integrable modules with finite dimensional weight spaces for twisted toroidal extended affine Lie algebra for nonzero level with nulity $\geq 3$. In \cite{[22]}, Yuly Billig and John Talboom gave a classification result for the category of jet modules for divergence zero vector fields on torus, which plays a pivotal role in the classification result of  twisted toridal extended affine Lie algebras. John Talboom gave a similar result for Hamiltonian vector fields on torus in \cite{[11]}. In \cite{[1]}, Rao studied structure theory of Hamiltonian extended affine Lie algebra and its integrable simple weight modules. In this article we consider the twisted form of Hamiltonian extended affine Lie algebra and study its integrable  simple weight modules. In order to classify all these modules we use a result of \cite{[11]}.  
\\\\
Let $\mathfrak{g}$ be a finite dimensional simple Lie algebra over $\mathbb{C}$ and $A_n$ be the  Laurent polynomial ring in $n\geq 2$ commuting variables $t_1, \dots ,t_n$. Let $L(\mathfrak{g})=\mathfrak{g}\otimes A_n$ be a loop algebra. Let $L(\mathfrak{g})\oplus \Omega_{A_n}/dA_n$ be the universal central extension of $L(\mathfrak{g})$. This is called the toroidal Lie algebra (TLA). Let $Der(A_n)$ be the Lie algebra of all the derivations of $A_n$. Then $L(\mathfrak{g})\oplus \Omega_{A_n}/dA_n\oplus Der(A_n)$ is called the full toroidal Lie algebra (FTLA). TLA and FTLA are not EALA, as they do not admit any non-degenerate symmetric invariant bilinear form. So instead of $Der(A_n)$, one takes ${S}_n$ the subalgebra of $Der(A_n)$, consisting of divergence zero vector fields. The Lie algebra  $L(\mathfrak{g})\oplus \Omega_{A_n}/dA_n\oplus S_n$ is called a toroidal extended affine Lie algebra which admits a non-degenerate symmetric bilinear form and hence is an example of EALA. Now we take $n=2k$, a positive even integer and let $H_n$ be the classical Hamiltonian Lie subalgebra of Der($A_n$), then  a quotient algebra of $L(\mathfrak{g})\oplus \Omega_{A_n}/dA_n\oplus H_n $ becomes an extended affine Lie algebra, called as Hamiltonian extended affine algebra. \\\\
 Here we consider more general EALAs. We take $\sigma_1, \dots ,\sigma_n$ to be finite order commuting automorphisms
  of $\mathfrak{g}$ and consider the multi-loop algebra $L(\mathfrak{g},\sigma)=\oplus_{r \in \mathbb{Z}^n}\mathfrak{g}(\underline{r})\otimes t^k$. Corresponding to a diagram automorphism, Batra studied finite dimensional modules of twisted multiloop algebras in \cite{[23]}. The finite dimensional representations of multiloop algebras (both untwisted and twisted)  have been studied by Lau in \cite{[18]}. We assume that this multi-loop algebra is a Lie torus. Now we consider the universal central extension of multi-loop algebra and add $H_n$, the classical Hamiltonian Lie algebra. We quotient this resulatant Lie algebra by $K(m)$ (See definition of $K(m)$ above Proposition 2.1). We denote it by $\tau$. In this paper we will classify irreducible integrable modules with finite dimensional weight spaces for 
 $\tau$, where the zero degree central operators act non-trivially for any $n\geq 2$. We make use of some important results from \cite{[1]} and \cite{[11]} in order to classify our modules.\\\\
  The paper has been organized as follows. In Section \ref{sec 2},  we recall the definition of multiloop algebra, Lie torus and classical Hamiltonian Lie algebra and then construct the twisted Hamiltonian extended affine Lie algebra. Towards the end of Section \ref{sec 2}, we recall a Proposition about the dimension of graded component of central extension part. In Section \ref{sec 3}, we fix a Cartan subalgebra and with respect to this Cartan subalgebra we give a root space decomposition of $\tau$. We define integrable modules for this Lie algebra and give a triangular decomposition of $\tau$. Now we are taking a different  triangular decomposition of $\tau$ from the triangular decomposition of toroidal or full toroidal Lie algebra  for the level zero or non zero case. With respect to this triangular decomposition we prove that any irreducible integrable weight module is a highest weight module. Then the highest weight space becomes an irreducible module for the zeroth part of triangular decomposition of $\tau$. In fact it becomes $\Z^{n-2}$-graded. We also show that the whole central extension part of zeroth part acts trivially on the highest weight space. This was not the case for twisted full toroidal Lie algebra. In Section \ref{sec 4}, we study the highest weight space in detail. We take subalgbera $L$ of $\tau^0$ and a subspace $W$ of $M$, and we prove that $W$ is an $L$ submodule of $M$. We also prove that $M/W$ is completely reducible module for $L$ and we identify each irreducible component of $M/W$ with tensor product of irreducible module for a semisimple Lie algebra and symplectic algebra. Finally we prove our main Theorem \ref{Thm 6.6}.  
\section{Notation and Preliminaries}\label{sec 2}
Let $n=2k$ be a positive integer and consider $A_{n}=\C[t_1 ^{\pm 1},t_2^{\pm 1}\dots , t_n^{\pm 1}]$ be a Laurent polynomial
ring in $n$ commuting variables $t_1, \dots ,t_n$ . Let $\mathfrak{g}$ be a finite dimensional simple Lie algebra over $\mathbb{C}$ with a Cartan
subalgebra $\mathfrak{h}$ and $\sigma_1,\sigma_2,\dots, \sigma_n$ be commuting
automorphisms of $\mathfrak{g}$ of orders $m_1,m_2,\dots, m_n$ respectively. Also let $(|)$ be the standard Killing form on $\mf g$.  For
$m=(m_1,m_2,\dots ,m_n)\in \mathbb{Z}^n$ we define $\Gamma=m_1\mathbb{Z}\oplus \dots
\oplus m_{k-1}\mathbb{Z}\oplus m_{k+1}\mathbb{Z}\oplus \dots \oplus m_{2k-1}\mathbb{Z}$, $\Gamma_0=m_k\mathbb{Z}\oplus m_{2k}\mathbb{Z} $ and $\Lambda:=\mathbb{Z}^{n-2}/\Gamma$. Then $m^\prime=(m_1,\dots, m_{k-1},m_{k+1},\dots, m_{2k-1})\in \Gamma$. Set $\bar{\Gamma}=m_1\mathbb{Z}\oplus \dots
\oplus m_{2k}\mathbb{Z}$. For convenience we write $\Gamma\oplus \Gamma_0=$  $\bar{\Gamma}$ throughout the paper. We denote elements of $\Z^n$ by  $r=(r_1,\dots, r_n)$.
 Then we have  $\mathfrak{g}=\displaystyle{\bigoplus_{\underline{r}\in {\mathbb Z^n}/{\bar \Gamma}}}\mathfrak{g}(\underline{r}) $, where  $\mathfrak{g}(\underline{r})=\{X \in \mathfrak{g}|\sigma_i(X)=\zeta_{i}^{r_i}X,1\leq i\leq n \}$, $\zeta_i$
are $m_i$-th primitive root of unity for $i=1,\dots,n$ and $\underline r$ denote the image of $r=(r_1, \dots, r_n)$  in  $\Z^n/{\bar \Gamma}$.
\\
Define the multiloop algebra $L(\mathfrak{g},\sigma)=\displaystyle{\bigoplus_{{r}\in {\mathbb Z^n} }}\mathfrak{g}(\underline{r})\otimes t^r$ with the Lie brackets $$[x\otimes t^a, y\otimes t^b]=[x,y]\otimes t^{a+b},$$
for $x\in \mathfrak{g}(\underline{a}),y\in \mathfrak{g}(\underline b)$ and $a,b\in \mathbb{Z}^n$.
\\
Now let $\mathfrak{g}_1$ be any arbitrary
finite-dimensional simple Lie algebra over $\mathbb{C}$ with a Cartan subalgebra
$\mathfrak{h}_1$. Let $\Delta(\mathfrak{g}_1,\mathfrak{h}_1)=supp_{\mathfrak{h}_1}(\mathfrak{g}_1)$. Then
$\Delta_1^{\times}=\Delta^{\times}(\mathfrak{g}_1,\mathfrak{h}_1)=\Delta(\mathfrak{g}_1,\mathfrak{h}_1) - \{0\}$ is an
irreducible reduced finite root system with at most two root lengths. Define

\[ \Delta_{1,en}^{\times}= \begin{cases} \Delta_1^{\times}\cup 2\Delta_{1,sh}^{\times}  &
 \text{if $\Delta_1^{\times}=B_l$ types}\\
 
\Delta_1^{\times} &\text{ otherwise}.\end{cases}\]
 
\begin{definition} 
A finite dimensional $\mathfrak{g}_1$-module $V$ is said to satisfy condition $(M)$ if $V$ is
irreducible with dimension greater than 1 and weights of $V$ relative to $\mathfrak{h}_1$ are
contained in $\Delta_{1,en}^{\times}$.
\end{definition} 

\begin{definition}
A multiloop algebra $L(\mathfrak{g},\sigma)$ is called a
Lie torus ($LT$) if 
\begin{itemize}
\item[(1)] $\mathfrak{g}(\underline{0})$ is a finite dimensional simple Lie algebra.
\item[(2)] For $\underline{r}\neq 0$ and $\mathfrak{g}(\underline{r})\neq 0$, $\mathfrak{g}(\underline{r})\cong
U(\underline{r})\oplus W(\underline{r})$, where $U(\underline{r})$ is trivial as $\mathfrak{g}(\underline{0})$-module
and either $W(\underline{r})$ is zero or satisfy condition (M).
\item[(3)] The order of the group generated by $\sigma_i,1\leq i \leq n$ is equal to
the product of the orders of each $\sigma_i$, for $1\leq i\leq n$.
\end{itemize}
   
\end{definition}

Let $A_n(m)=\mathbb{C}[t_1^{\pm  m_1},\dots,t_n^{\pm  m_n} ]$ and $\Omega_{A_n(m)}$ be a vector space spanned by the symbols
$t^rK_i,1\leq i\leq n,r\in \bar \Gamma$. Also assume that $dA_n(m)$ be the subspace of $\Omega_{A_n(m)}$ spanned by
$\sum_{i=1}^{n} r_it^rK_i,$ $r\in \bar \Gamma$. Set $Z=\Omega_{A_n(m)}/dA_n(m)$, a typical element of $Z$ is given by $K(u,r)=\displaystyle{\sum _{i=1}^n}u_it^rK_i$, for $u=(u_1,\dots ,u_n)\in \mathbb{C}^n$ and $r\in \bar \Gamma$. For convenience, We denote the Lie torus $L(\mathfrak g, \sigma)$ by $LT$. It is well known that $$\overline{LT}=LT \oplus
\Omega_{A_n(m)}/dA_n(m) $$ is the universal central extension of ${LT}$ with the following
brackets:
\begin{align}\label{a2.1}
[x(p),y(q)]=[x,y](p+q)+(x|y)K(p,p+q),
\end{align}
 
   where $x(p)=x\otimes t^p$, $p,q \in \mathbb{Z}^n$. \\
  Note that for $x\in \mathfrak{g}(\underline{p}), y\in \mathfrak{g}(\underline{q})$ with $(x|y)\neq 0$ we have $p+q\in \Bar{\Gamma}$, so the Lie bracket \ref{a2.1} is well defined.
  Let  $Der(A_n(m))$ be the space of all derivations of $A_n(m)$. A basis for $Der(A_n(m))$
is given by \{$d_i,t^rd_i|1\leq i\leq n,0 \neq r \in \bar \Gamma$\}, where $d_i=t_i\frac{d}{dt_i}$. For $u=(u_1,\dots ,u_n)\in \mathbb{C}^n, r\in \bar \Gamma $, set $D(u,r)=\displaystyle{\sum _{i=1}^n}u_it^rd_i$.  $Der(A_n(m))$ forms a Lie algebra with the Lie brackets 
$$[D(u,r),D(v,s)]=D(w,r+s),$$
where $u,v\in \mathbb{C}^n, r,s \in \bar \Gamma$, $w=(u,s)v-(v,r)u$ and $(u,s),(v,r)$ denotes the standard inner product on $\C^n.$
 \\
 Now $Der(A_n(m))$ acts
on $\Omega_{A_n(m)}/dA_n(m)$ by
 $$D(u,r).K(v,s)=(u,s)K(v,r+s)+(u,v)K(r,r+s).$$
It is well known that $Der(A_n(m))$ admits an abelian extension of $\Omega_{A_n(m)}/dA_n(m)$ with the Lie brackets
\begin{align}\label{a2.2}
[D(u,r),K(v,s)]=(u,s)K(v,r+s)+(u,v)K(r,r+s),
\end{align}
\begin{align}\label{a2.3}
[D(u,r),D(v,s)]=D(w,r+s)+(u,s)(v,r)K(r,r+s),
\end{align}

where $u,v\in \mathbb{C}^n, r,s \in \bar \Gamma$ and $w=(u,s)v-(v,r)u$.\\
For $s=(s_1\dots ,s_n)\in\bar \Gamma$ we define $\bar{s}=(s_{k+1},\dots,s_{2k},-s_1,\dots, -s_k).$
Now consider a Lie subalgebra of $Der(A_n(m))$ given by 
$$H_n(m)=span\{d_i,h_r=D(\bar r,r)|1\leq i\leq n,r\in \bar \Gamma\}.$$ It is easy to see that $[h_r,h_s]=(\bar r,s)h_{r+s}$. This Lie algebra is known as Hamiltonian Lie algebra (See \cite{[1]},\cite{[11]}). Clearly $H_n(m)$ induces an action on $Z$. Let $\bar \tau=\overline {LT}\oplus H_n(m)$, this becomes a Lie algebra with the brackets (\ref{a2.1}), (\ref{a2.2}), (\ref{a2.3}) and
$$[h_r,x\otimes t^s]=(\bar r,s)x\otimes t^{r+s},$$ for all $x\in \mathfrak{g}(\underline s), r \in \bar \Gamma, s \in \mathbb Z^n$. \\
Define $K(m)=span\{K(u,r)|u\in \mathbb{C}^n, r\in \bar{\Gamma} \setminus \{0\}, (u,\bar r)=0\}$. It is easy to see that $K(m)$ is an ideal of $\bar{\tau}$. 
\begin{proposition}(\cite{[1]}, Proposition 3.1)
\begin{enumerate}
\item $Z/K(m)$ is $\bar{\Gamma}$-graded with dim$(Z/K(m))_r=1$ having basis element $K(\bar r,r),$ for $r\neq 0.$
\item dim$(Z/K(m))_0=n$ with basis $K_i, 1\leq i\leq n$.
\end{enumerate} \qed
  \end{proposition}

Define $\tau_n =\tau=LT\oplus Z/K(m)\oplus H_n(m)$ and construct a bilinear form on $\tau$ by
\begin{center}
$(x(r)|y(s))=\delta_{r,-s}(x|y),$ $\forall x\in \mathfrak{g}(\underline r) ,y\in \mathfrak{g}(\underline s),r,s\in \mathbb{Z}^n;$\\

$(h_r|K(\bar s,s))=\delta_{r,-s}(\bar r,\bar s), $ where $r,s\in \bar \Gamma \setminus \{0\}$.\\
$(d_i,K_j)=\delta_{i,j}$ for $1\leq i,j\leq n$.
\end{center}
All other brackets of bilinear form are zero. We see that $\tau$ becomes an extended affine Lie algebra and is called a twisted Hamiltonian extended affine Lie algebra (Twisted HEALA). \\

Let $\mathfrak{h}(\underline{0})$ denote a Cartan subalgebra of $\mathfrak{g}(\underline{0})$. Then by \cite{[13]} Lemma
3.1.3, $\mathfrak{h}(\underline{0})$ is ad-diagonalizable on $\mathfrak{g}$ and $\bigtriangleup
^\times=\bigtriangleup^\times (\mathfrak{g},\mathfrak{h}(\underline{0}))$ is an irreducible finite root system in
$\mathfrak{h}(\underline{0})$ (Proposition 3.3.5,\cite{[13]}). Let $\bigtriangleup_0
:=\bigtriangleup(\mathfrak{g}(\underline{0}),\mathfrak{h}(\underline{0}))$. It is known from (\cite{[13]}, Lemma 3.1.3 and Proposition 3.3.5) that $\mathfrak h(\underline 0)$ is a ad-diagonalizable subalgebra of $\mathfrak g$ and $\Delta_{\mathfrak g}^{\times}=Supp_{\mathfrak h(\underline 0)}(\mathfrak g) \setminus \{0\}$ is an irreducible (possibly
non-reduced) finite root system in $\mathfrak h^*$. One of the main properties of Lie tori is
that $\Delta_{\mathfrak g}=\Delta_{\mathfrak g}^{\times} \cup \{0\}=\bigtriangleup_{0,en} $ (Proposition 3.2.5,\cite{[2]}). Let $\pi=\{\alpha_1,\alpha_2,...,\alpha_d\}$ be the simple roots of $\bigtriangleup_0.$ Let $Q$ be the root system of $\bigtriangleup_0$ and $Q^+=\displaystyle{\bigoplus_{i=1}^{d}}\mathbb Z_+\alpha_i$. Here $\mathbb Z_+$ denotes the set of non-negative integers.\\
 Define $\mathfrak{g}(\underline{r},\alpha):=\{x\in \mathfrak{g}(\underline{r})|[h,x]=\alpha(h)x,\forall h\in \mathfrak{h}(\underline{0})\}$. By the properties of Lie torus $\mathfrak{g}(\underline r)=\displaystyle{\bigoplus_{\alpha \in \mathfrak{h}(\underline 0)^*} }\mathfrak{g}(\underline{r},\alpha) $.
    \begin{remark}\label{r1}
$\overline {LT}$ is a Lie $\mathbb Z^n$ torus of type $\bigtriangleup_{0,en}. $ Then by \cite{20}, Definition 4.2, $\overline {LT
}$ is generated as Lie algebra by $(\overline {LT})_\alpha =\displaystyle{\bigoplus_{r\in \mathbb Z^n}}\mathfrak{g}(\underline r, \alpha)\otimes t^r $, $\alpha \in (\bigtriangleup_{0,en})^\times$.
\end{remark}

\section{Existence of Highest weight space}\label{sec 3}
In this section we will give a root space decomposition of $\tau$. Let
$ {H}=\mathfrak{h}(\underline{0}) \displaystyle{\bigoplus _{i=1}^n} \mathbb{C}K_i \bigoplus _{i=1}^{n}\mathbb{C}d_i$
be our Cartan subalgebra for the root space decomposition of $\tau$. Define $\delta
_i,w_i \in H^{*}$ by setting 
\begin{center}
$\delta_i(\mathfrak{h}(\underline{0}))=0,\delta_i(K_j)=0$ and $\delta_i(d_j)=\delta_{ij}$, and
\end{center}
\begin{center}
$w_i(\mathfrak{h}(\underline{0}))=0$,$w_i(K_j)=\delta_{ij}$ and $w_i(d_j)=0$.
\end{center}

Take $\delta_{\beta}=\sum_{i=1}^n \beta_i\delta_i$ for $\beta \in \mathbb{C}^n$. For
$r\in \mathbb{Z}^n$, we shall refer to the vector $\delta_{r+\gamma}$ as the
translate of $\delta_r$ by the vector $\delta_{\gamma}$, where $\gamma \in
\mathbb{C}^n$.
 Then  we have $\tau =\displaystyle{\bigoplus_{\beta\in
\bigtriangleup }}\tau_{\beta}$, where $\bigtriangleup \subseteq
\{\alpha+\delta_r|\alpha\in \bigtriangleup_{0,en},r\in \mathbb{Z}^n\}$,\\
\begin{center}
$\tau_{\alpha+\delta_r}=\mathfrak{g}(\underline{r},\alpha)\otimes t^r$ for $\alpha\neq 0$, \\
$\tau_{\delta_r}=\mathfrak{g}(\underline{r},0)\otimes t^r\oplus
 \mathbb{C}K(\bar r,r)\oplus\mathbb{C} h_r $ for
$0\neq r \in \bar \Gamma  $, \\
$\tau_{\delta_r}=\mathfrak{g}(\underline{r},0)\otimes t^r$ for $r\in \mathbb{Z}^n\setminus \bar \Gamma$ and $\tau_0=H$.
\end{center}
We call elements of $\bigtriangleup$ roots of $\tau$.
A root $\beta=\alpha+\delta_r$ is called a real root if $\alpha\neq 0$. Let $\bigtriangleup
^{re}$ denote the set of all real roots and
$\beta^{\vee}=\alpha^{\vee}+\frac{2}{(\alpha|\alpha)}\displaystyle{\sum_{i=1}^n }r_iK_i$ be co-root of $\beta$, where $\alpha^{\vee}$ is the co-root of $\alpha\in
\bigtriangleup_{0,en}$. For $\gamma\in \bigtriangleup^{re}$, define
$r_{\gamma}(\lambda)=\lambda-\lambda(\gamma^{\vee})\gamma$ for $\lambda\in H^*$. Let
$ {\Omega}$ be the Weyl group of $\tau$ generated by $r_{\gamma}, \forall \gamma\in
\bigtriangleup^{re}$. 
\begin{definition}
A $\tau$ -module $V$ is called integrable if 
\begin{itemize}
\item[(1)] $V=\bigoplus_{\lambda\in H^*}V_{\lambda}$, where $V_{\lambda}=\{v\in
V|h.v=\lambda(h)v,$   $ \forall\, h\in H\}$ and $dim(V_{\lambda}) < \infty$.
\item[(2)] All the real root vectors act locally nilpotently on $V,$ i.e.,
$\mathfrak{g}(\underline{r},\alpha)\otimes t^r$ acts locally nilpotently on $V$ for all $0\neq \alpha\in
\bigtriangleup_{0,en}$.
\end{itemize}
\end{definition}
We have the following Proposiotion as \cite{[5]}.
\begin{proposition}\label{prop 3.1}(\cite{[5]})
Let $V$ be an irreducible integrable module for $\tau$. Then
\begin{itemize}
\item[(1)] $P(V)=\{\gamma \in H^*|V_{\gamma}\neq 0\}$ is $\Omega$- invariant.
\item[(2)] $dim(V_{\gamma})=dim (V_{w\gamma}), \forall w\in \Omega$.
\item[(3)] If $\lambda \in P(V)$ and $\gamma\in \bigtriangleup^{re}$, then
$\lambda(\gamma^{\vee})\in \mathbb{Z}$.
\item[(4)] If $\lambda \in P(V)$ and $\gamma\in \bigtriangleup^{re}$, and
$\lambda(\gamma^{\vee})> 0$, then $\lambda-\gamma\in P(V)$.
\end{itemize}
\end{proposition}
 Now consider the triangular decomposition of $\tau$ 
 \begin{center}
 $\tau_{++}=span\{X(r),h_s,K(\bar s,s):r\in \mathbb{Z}^n,s\in \bar \Gamma,X\in \mathfrak{g}(\underline r),r_k>r_{2k},s_k>s_{2k}\},$
$$\tau_{+}=span\{X_{\alpha}(r),h_s,K(\bar s,s):r\in \mathbb{Z}^n,s\in \bar \Gamma,X_\alpha\in \mathfrak{g}(\underline r,\alpha),r_k=r_{2k}>0$$ \\
$\; or\; r_k=r_{2k}=0, \alpha>0,\;s_k=s_{2k}>0\},$
 \end{center}
 $\tau^0=span\{X( r), h_s,K(\bar s,s):r\in \mathbb{Z}^n, s\in \bar \Gamma, X\in \mathfrak{g}(\underline r,0), r_k=r_{2k}=0,s_k=s_{2k}=0\}$,
  \begin{center}
 $\tau_{--}=span\{X(r),h_s,K(\bar s,s):r\in \mathbb{Z}^n,s\in \bar \Gamma,X\in \mathfrak{g}(\underline r),r_k<r_{2k},s_k<s_{2k}\},$
$$\tau_{-}=span\{X_{\alpha}(r),h_s,K(\bar s,s):r\in \mathbb{Z}^n,s\in \bar \Gamma,X_\alpha\in \mathfrak{g}(\underline r,\alpha),r_k=r_{2k}<0$$ \\
$\; or\; r_k=r_{2k}=0, \alpha<0,\;s_k=s_{2k}<0\}.$
 \end{center}
 Now we define $\tau^+=\tau_{++}\oplus \tau_+$ and $\tau^-=\tau_{--}\oplus \tau_-$, therefor $\tau=\tau^-\oplus \tau^0\oplus \tau^+$ is the triangular decomposition of $\tau$. Note that 
 $[\tau_{++},\tau_-\oplus \tau_+]\subset  \tau_{++},\; [\tau_{--},\tau_-\oplus \tau_+] \subset \tau_{--}$.\\
 \begin{remark}\label{r2}
We see that $LT_{n-2} =span\{X_\alpha(r), X(r), K(\bar{s},s): \alpha\in \bigtriangleup_{0,en}, r_k=r_{2k}=s_k=s_{2k}=0\}$ forms a Lie torus from the Lie algebra $\mathfrak{g}$ with automorphisms $\sigma_1,\dots,\sigma_{k-1},\sigma_{k+1},\dots ,\sigma_{2k-1}$.
 \end{remark}
 We shall now define automorphism of full toroidal Lie algebra  $FT=\mathfrak{g}\otimes A_n\oplus \Omega A_n/dA_n\oplus Der(A_n)$, where $\Omega A_n/dA_n$ can be defined as quotient space of $\Omega A_n=span\{t^r
 K_i:r\in \mathbb{Z}^n, 1\leq i\leq n\}$ by $d{A_n}=span\{\sum_{i=1}^nr_it^rK_i\}$. Let $GL(n,\mathbb{Z})$ be the group of invertible matrices with integer coefficients. Let $B\in GL(n,\mathbb{Z})$ which acts on $\mathbb{C}^n$, denote this action by $Bu,\; u\in \mathbb{C}^n$. Now we define an automorphism of $FT$, again denote it by $B$ by,
 $$B.X(r)=X(Br),\; B.K(u,r)=K(Bu,Br),\;B.D(u,r)=D(Fu,Br),$$
 where $F=(B^T)^{-1}$. \\
 Define $$K_B=span\{K(Bu,Br)|(u,\bar r)=0, u\in \mathbb{C}^n,r\in \bar \Gamma\}$$ $$ H_B=span\{d_i,D(F\bar{r},Br)|1\leq i\leq n,\;r\in \bar \Gamma\setminus\{0\}\}$$
 We see that $[H_B,K_B]\subseteq K_B$. Now take $\tau_B=LT_B\oplus Z/K_B\oplus H_B$, where $LT_B$ is the image of $LT$ under the automorphism $B$ of $FT$. Now under this automorphism $\tau\cong \tau_B$. In order to avoid notational complexity instead of working with $\tau_B$, we will work with $\tau$ after applying $B$ also. \\
 Let $V$ be an irreducible integrable module with finite dimensional weight spaces with respect to the Cartan subalgebra $H$, on which not all $K_i$ act trivially. Then after twisting the module by a suitable isomorphism of above type, we can assume that only $K_k$ acts non-trivially on the module and $K_i$ acts trivially for $1\leq i\leq n $ and $i\neq k$.
 \begin{proposition}\label{p3.2}
 Suppose $V$ is an irreducible integrable module for $\tau$ with finite dimensional weight spaces with respect to the Cartan subalgebra $H$. If the central element $K_k$ acts as positive integer and $K_i$ acts trivially for $1\leq i\leq n$ and $i\neq k$, then there exists a weight $\lambda$ such that 
 \begin{enumerate}
 \item $X_\alpha(r).V_\lambda=0$, where $r\in \mathbb{Z}^n,\;X_\alpha\in \mathfrak{g}(\underline r,\alpha),\;r_k>0$.
 \item $h_s.V_\lambda=K(\bar s,s).V_\lambda=0$, where $s\in \bar \Gamma,\;s_k>0$
 \end{enumerate}
  \end{proposition}
\begin{proof}
We can prove this Proposition similarly as Theorem 5.2 of \cite{[4]}. Note that instead of zero-th coordinate we need to work out with $k$-th coordinate here. 
\end{proof}

Let	
\begin{equation*}
B_{n,n} = 
\begin{pmatrix}
1 & 0& \cdots & 0&\cdots  & 0 \\
0 & 1 & \cdots & 0&\cdots & 0 \\
\vdots  & \vdots  & \ddots & \vdots& \ddots & \vdots  \\
0 & 0 & \cdots & b_{k,k}&\cdots & b_{k,2k}\\
\vdots &\vdots &\ddots &\vdots& \ddots & \vdots \\
0 & 0 & \cdots & b_{2k,k}&\cdots & b_{2k,2k}
\end{pmatrix}
\end{equation*}
 and \[ \left( \begin{array}{ccc}
b_{k,k} & b_{k,2k} \\
	b_{2k,k} & b_{2k,2k}  \end{array} \right) = \left( \begin{array}{ccc}
a & 1 \\
	a-1 & 1  \end{array} \right) \] with $2a-1>0$, all diagonal entries of $B_{n,n}$ are 1, rest of all other entries are zero except $b_{k,k},b_{k,2k},b_{2k,k},b_{2k,2k}$. 
	 Now we twist the module by an isomorphism $B\in GL(n,\mathbb{Z})$. 
Note that after twisting the module by $B_{n,n}$, we have a weight space $V_\lambda$ of $V$ such that $V_{\lambda+\eta +\delta_s}=0$, where $\eta \in Q$, $s_k-s_{2k}>0$, by Proposition \ref{p3.2}.\\
Define an ordering on the $H^*$ by $\lambda\leq \mu$ in $H^*$ if and only if $\mu -\lambda=\displaystyle{\sum_{i=1}^{d} n_i\alpha_i}+n_{d+k}\delta_k+n_{d+2k}\delta_{2k}$ where $n_i \in \mathbb Z$, either $(i)$ $n_{d+k}-n_{d+2k}>0$, or $(ii)$ $n_{d+k}=n_{d+2k}>0$ or $(iii)$ $n_{d+k}=n_{d+2k}=0$ and $n_i \in \mathbb Z_+$, $1\leq i \leq d$. By an ordering of $\mathfrak{h}(\underline{ 0})^*$ we mean that the oredering "$\leq$" restricted to $\mathfrak{h}(\underline{ 0})^*$.
\begin{theorem}
Let $V$ be an irreducible integrable module for $\tau$ with finite-dimensional weight spaces with respect to the Cartan subalgebra $H$, then there exists a weight space $V_\mu$ of $V$ such that $\tau^+.V_\mu=0$.
\end{theorem}
\begin{proof}
We consider the Lie algebra $\tilde{g}=span \{X(r), K(\bar s,s),h_s,K(u,0),D(u,0):X\in \mathfrak{g}(\underline{r}), r\in \mathbb{Z}^n,\; s\in \bar{\Gamma}, r_k=r_{2k},s_k=s_{2k}, u\in \mathbb{C}^n\}$. Consider the subspace $$W=\displaystyle{ \bigoplus_{\eta \in Q,\;l\in \mathbb{Z}^n,\;l_k=l_{2k}}} V_{\lambda+
\eta+\delta_l}.$$ We can see that $W$ is an integrable (may not be irreducible) $\tilde{g}$-module with respect to the same Cartan subalgebra $H$. \\
Fix a weight $\lambda_1=\lambda +\eta+\delta_l$ of $W$. Let $g_i=\lambda_1(d_i)$ for $1\leq i \leq n$ and set $g=(g_1,g_2,..,g_n)$. Consider $P_g(W)=\{ v \in W: d_i.v=g_iv, \; 1\leq i \leq n\}$. It is easy to see that $P_g(W)$ is an integrable module for $\mathfrak{g}(\underline{ 0})$ with finite dimensional weight spaces. Now by Lemma 3.5 of (\cite{[19]}), we have $P_g(W)$ a finite dimensional module for $\mathfrak{g}(\underline{ 0})$. Hence $P_g(W)$ has only finitely many weights and consider the maximal weight $\lambda_2$ corresponding to $\lambda_1$ with respect to the ordering $"\leq "$. So $\lambda_1 \leq \lambda_2$ and $\lambda_2+\eta$ is not a weight of $W$, for all $\eta > 0$ and $\eta \in \bigtriangleup_{0,en}.$ Now we can use the method used in Theorem 5.2 of \cite{[4]} to prove that there exists a weight $\mu=\lambda +\eta'+\delta_r$ of $W$, where $\eta' \in Q, r \in \mathbb Z^n, r_k=r_{2k}$ such that $V_{\mu +\eta+\delta_s}=0$ for $\eta \in Q, s \in \mathbb Z^n, s_k=s_{2k}>0$ or $s_k=s_{2k}=0$ but $\eta \in Q^+$. Now we can prove $\tau^+.V_\mu=0$ with the help of similar method used in Theorem 5.3 of \cite{[1]}. 
\end{proof}
The following Proposition is standard.
\begin{proposition}\label{prop 3.3}
\begin{enumerate}
\item $M=\{v \in V: \tau^+.v=0 \}$ is a non-zero irreducible $\tau^0$ module.
\item The weights of $M$ lies in the set $\{ \mu + \delta_r:r_k=r_{2k}=0, r \in \mathbb Z^n \},$ for a fixed weight $\mu$ of $M$.
\end{enumerate} \qed
\end{proposition}
\begin{proposition}
If $\mu (h)=0,\; \forall\; h\in \mathfrak{h}(\underline{0})$ and for some weight $\mu$ of $M$, then $M$ becomes an $H_{n-2}(m^\prime)$ irreducible module.
\end{proposition}
\begin{proof}
Suppose $\alpha$ be a positive root of $\bigtriangleup_{0,en}$ such that $\mathfrak{g}(\underline{ r}, -\alpha) \neq 0$. Let $Y_\alpha \otimes t^r \in \mathfrak{g}(\underline{ r}, -\alpha)\otimes t^r$ acting non-trivially on $v_\mu $ for some weight vector $v_\mu $ of $M$. Then $\mu -\alpha +\delta_r$ is a weight of $V$. Hence by Proposition \ref{prop 3.1}(3), we have $r_\alpha(\mu -\alpha +\delta_r)=\mu -\alpha +\delta_r - (\mu -\alpha +\delta_r)(\alpha^\vee) \alpha=\mu +\alpha +\delta_r$ is a weight of $V$, a contradiction. By Remark \ref{r1} and Remark \ref{r2}, $(LT\oplus Z/K(m))\cap \tau^0$ acting trivially on $v_\mu$. Consider the non-zero subspace $W=\{v \in M: (LT\oplus Z/K(m))\cap \tau^0.v=0\}$ of $M$. We see that $W$ is a $\tau^0$-module, hence by irreducibility $W=M$. Note that $d_k,d_{2k}$ are central in $\tau^0$ and hence acts as scalar on $M$. Therefore $M$ is irreducible module for $H_{n-2}(m^\prime)$.
\end{proof}
Classification of irreducible weight modules for $H_n$ is still unknown, So we can not comment on irreducible weight modules of $H_{n-2}$. Now throughout the paper we assume that $\bar \mu=\mu|_{\mathfrak{h}(\underline{0})}\neq 0$. Therefore there exists a simple root $\alpha\in \bigtriangleup_0$ such that $\mu(h_\alpha)\neq 0$.\\
 \indent  We fix some $i,\; 1\leq i\leq n,\; i\neq k,2k$ and consider the extended loop algebra $\mathfrak{g}(\underline{0})\otimes \mathbb{C}[t_i^{\pm m_i}]\oplus \mathbb{C}d_i$. Let $\theta$ be the highest root of $\mathfrak{g}(\underline{0})$, $\theta^{\vee}$ be its co-root and $\Omega_0$ be the Weyl group of $\mathfrak{g}(\underline{0})$. Also assume that $\Omega_i$ be the Weyl group of the loop algebra $\mathfrak{g}(\underline{0})\otimes \mathbb{C}[t_i^{\pm m_i}]\oplus \mathbb{C}d_i$. We can see that $h.v={\mu} (h)v$ for all $v\in M$  and $h\in \mathfrak{h}(\underline{0})$. Define $t_{i,h}:(h(\underline{ 0})\oplus \mathbb{C}d_i)^*\rightarrow (h(\underline 0)\oplus \mathbb{C}d_i)^* $ by $t_{i,h}(\lambda)=\lambda-\lambda(h)\delta_i$, where $h\in \mathbb{Z}(\Omega_0\theta^\vee) $. Note that $\mathbb{Z}(\Omega_0\theta^\vee)$ lies in $\mathbb{Z}$ linear combination of coroots for $\bigtriangleup_0$ and $\lambda(\theta^\vee)>0$ for all weight $\lambda$ of $M$. Therefore $r_{\mu}=min_{h\in \mathbb{Z}(\Omega_0\theta^\vee)}\{\mu(h):\mu(h)>0\}\in \mathbb{N}$. Let $r_{\bar \mu}=\mu(h_0)$ for some $h_0\in \mathbb{Z}(\Omega_0\theta^\vee)$. It is well known that $\Omega_i$ is semidirect product of $\Omega_0$ and the commutative group generated by $t_{i,h}$, $h\in\mathbb{Z}(\Omega_0\theta^\vee)$.  
\begin{lemma}\label{lem 3.1}
For all $s_i\in \mathbb{Z}$, there exists $w\in \Omega_i$ such that $w({\mu}+s_i\delta_i)={\mu}+\bar{s_i}\delta_i$, where $0\leq \bar{s_i}< r_{{\mu}}$. 
\end{lemma}
\begin{proof}
Let $s_i=\bar{s_i}+p_ir_{\mu}$. Then $t_{i,h_0}^{p_i}.({\mu}+s_i\delta_i)={\mu}+\bar{s_i}\delta_i$. 
\end{proof}
Let $A_{n-2}(m)$ be the Laurent polynomial ring with $(n-2)$ variables $t_i^{m_i}$ for $1\leq i\leq n$, $i\neq k,2k$. 
Now we consider the Lie algebra $ \tau_{n-2}=\mathfrak{g}(\underline{0})\otimes A_{n-2}(m)\oplus Z_{n-2}/K_{n-2}(m)\oplus H_{n-2}(m)$. Let $ \Omega (n-2)$ be the Weyl group of $\tau_{n-2}$, then $W_i\subseteq \Omega(n-2)$ for $1\leq i\leq n$, $i\neq k,2k$.  
\begin{lemma}\label{lem 3.2}
Let $\delta_r=\displaystyle{ \sum_{ 1\leq i\leq n, \, i\neq k,2k}} r_i\delta_i$, where $r=(r_1,\dots,r_{k-1},r_{k+1},\dots, r_{n-1})\in \mathbb{Z}^{n-2}$. Let $r_i=c_i+s_im_i$, where $0\leq c_i<m_i$. Then there exists $w\in \Omega(n-2)$ such that $w({\mu}+\delta_r)={\mu}+\delta_c+\underset{1\leq i \leq n,\;i\neq k,2k}\sum\bar{s_i}m_i\delta_i$, where $0\leq \bar{s_i}<r_{ {\mu}}$.
\end{lemma}
\begin{proof}
The proof follows from Lemma \ref{lem 3.1}.
\end{proof} 
Now from Proposition \ref{prop 3.3}, we see that $M$ is $\mathbb{Z}^{(n-2)}$-graded. Let $M=\oplus_{r\in \mathbb{Z}^{(n-2)}} M_r$, where $M_r=\{v\in M: d_i.v=(\mu(d_i)+r_i)v, \; 1\leq i\leq n,\; i\neq k,2k\}$. Let $N^\prime= \{M_r:r_i=c_i+s_im_i,\; 0\leq c_i<m_i,\; 0\leq s_i\leq r_{\mu}, 1\leq i\leq n, i\neq k,2k\}$, then by Lemma \ref{lem 3.1}, we see that weight spaces of $M$ are uniformly bounded by maximal dimension of elements of $N^\prime$. Take $N=\oplus _{M_r\in N^{\prime} } M_r$, then $N$ is finite dimensional. Now define $M^\prime(r)=\oplus_{r_i\leq k_i< m_i+r_i}M_k$, then $M=\oplus_{r \in
\Gamma}M^{\prime}(r)$ is $\Gamma$-garded $\tau^0$ module.\\

Now we give existence of an irreducible integrable module for $\tau_{n-2}$ such that at least one element of $M$ is injectively sits inside the module. 
We consider the module $V_1=U(\tau_{n-2})M$ for $\tau_{n-2}$. If every proper submodule of $V_1$ intersects $M$ trivially then take the quotient of $V_1$ by sum of all proper submodules, the quotient space is such an example. Now assume that $V_1$ has a proper $\tau_{n-2}$ submodule, say $W_1$ such that $W_1\cap M\neq 0$. Take $M_1=U(\tau_{n-2})(W_1\cap M)$, if every proper submodule of $M_1$ intersects $M$ trivially then construct the quotient space as above. In the quotient space $W_1\cap M$ goes injectively. Again if $M_1$ has a proper submodule, say $W_2$ such that $W_2\cap M\neq 0$, then construct $M_2=U(\tau_{n-2})(W_2\cap M)$.\\
\begin{lemma}
The decreasing chain of $\tau_{n-2}$-submodules $M_1\supseteq M_2\supseteq M_3\supseteq\dots $ terminates after a finite step.
\end{lemma}
\begin{proof}
Let $M_1\supsetneq M_2$ be two submodules of $\tau_{n-2}$. Then by Lemma \ref{lem 3.2}, we can find $N_i\subseteq N$ such that $M_i=U(\tau_{n-2})N_i$. Now this Proposition follows from Proposition 5.2.6 of \cite{sp}. 
\end{proof}
Let $M_i$ be such a minimal module. Then every proper submodule of $M_i$ intersect $M$ trivially. Let $\widetilde{M_i}$ be the quotient module of $M_i$ by the sum of all proper submodules. So $M_i\cap M$ sits injectively inside $\widetilde{M_i}$. Now $\widetilde{M_i}$ is an irreducible integrable module on which every $K_i$ acts trivially for $i\neq k,2k$.
 Now using Proposition (7.4) of  \cite{[1]}, we can find a vector $v\in \widetilde{M_i}$ such that $(Z/K(m))\cap \tau_{n-2}$ acts trivially on $v$. Now being an ideal of $\tau_{n-2}$ along with the irreducibility of $\widetilde{M_i}$, $(Z/K(m))\cap \tau_{n-2}$ acts trivially on $\widetilde{M_i}$. 
 \begin{theorem}
 $(Z/K(m))\cap \tau_{n-2}$ acts trivially on $M$.
 \end{theorem}
\begin{proof}
The proof follows from previous discussion and irreducibility of $M$ over $\tau^0$.
\end{proof}

Let $h_\alpha$ be as defined above, i.e. $\mu(h_\alpha)\neq 0$. Then we have the following Proposition. 
\begin{proposition}\label{prop 3.5}
\itemize 
\item[(1)] $h_{\alpha}\otimes t^k$ acts injectively on $M$ for every $k \in \Gamma$.
\item[(2)] $h_{\alpha}\otimes t^r.h_{\alpha}\otimes
t^s=\lambda_{r,s}h_{\alpha}\otimes t^{r+s}$ on $M$, where $\lambda_{r,s}=\lambda$
for all $r\neq 0,s\neq 0,r+s\neq 0$, $\lambda_{r,-r}=\mu$ for all $r\neq 0$ and
$\lambda_{0,r}=\bar{\lambda}(h_{\alpha})$ for all $r\in \Gamma$. Further we have
$\mu \lambda_{0,r}=\lambda^2\neq 0$.
\item[(3)] dim $ (M^{\prime}(r))=$ dim $ (M^{\prime}(s))$ for all $r,s \in
\Gamma$. 
\end{proposition} 
\begin{proof}
Follows from Theorem 9.1 of \cite{[1]}.
\end{proof}

\section{classification of integrable simple modules}\label{sec 4}

Now recall that our Lie algebra reduces to $\tau= LT\oplus H_n(m)$ with
$\tau^0=\displaystyle{\bigoplus_{\substack {r\in \mathbb{Z}^n\\ r_k=r_{2k}=0}}} \mathfrak{g}(\underline{r},0)\otimes t^r\oplus H_{n-2}(m')\oplus \mathbb{C}d_k\oplus \mathbb{C}d_{2k}$. Note that $d_k,d_{2k}$ are central in $\tau^0$, hence they act by scalars on $M$.  Take
$\mathfrak{g}^{\prime}=\{x\in \mathfrak{g} \;|\;[h,x]=0,\;\sigma_k(x)=\sigma_{2k}(x)=x,\: \forall\:h \in \mathfrak{h}(\underline{0})\}$. Now since
$\mathfrak{g}^{\prime}$ is invariant under $\sigma_i$'s (where $i\neq k,2k$ ), therefore $\mathfrak{g}^{\prime}$ is $\Lambda$
graded. It is easy to see that $L(\mathfrak{g}^{\prime},\sigma)=LT\cap \tau^0$ ($=LT^0$, say). Let us take
$H_{n-2}^{\prime}(m')=$ span$\{D(\bar{r},r)-D(\bar{r},0)| \;r \in \Gamma\}$.
We can easily check that $H_{n-2}^\prime (m')$ is a Lie subalgebra of $H_{n-2}(m')$.
Furthermore let us set $L=H_{n-2}^\prime(m') \ltimes L(\mathfrak{g}^\prime,\sigma')$ and $W=$
span$\{h_\alpha \otimes t^r.v-v|r\in \Gamma,v \in M\}$, where $\sigma'=(\sigma_1, \dots, \sigma_{k-1}, \sigma_{k+1}, \dots, \sigma_{2k-1})$. We can see that $W$ is an
$L$-module.
\begin{lemma}
\itemize
\item[(1)] $W$ is a proper $L$-submodule of $M$.
\item[(2)] $\widetilde{V}=M/W$ is a finite dimensional $L$-module.

\end{lemma}
\begin{proof}
 Let $z_i=h_{\alpha}\otimes t^m_i$
for each $i=1,\dots,n$ and $i\neq k,2k$. Without loss of generality we can assume that
$\lambda_{r,s}=1$ for $r\neq 0,s\neq 0,r+s\neq 0$. Therefore we can say that $W=$
span$\{z_i.v-v|v\in M\}$. Now same as Proposition 5.4(3) of \cite{[19]}, we can find a proper $LT^0$-submodule of $M$, which contains $W$.
\end{proof}
Let $\beta_i=\mu(d_i)$ for $1\leq i\leq n$ and $i\neq k,2k$. Then $\beta=(\beta_1\dots,\beta_n)\in\mathbb{C}^{n-2}$. Then for any $L$ module
$V^\prime$, we can give a $\tau^0$ module structure on $L(V^\prime)=V^\prime\otimes A_{n-2}$ by 
\begin{center}
$x\otimes t^k.(v_1\otimes t^s)=((x\otimes t^k).v_1)\otimes t^{k+s}$.
\end{center}
\begin{center}
$D(\bar{r},r).(v_1\otimes t^s)=((D(\bar{r},r)-D(\bar{r},0)).v_1)\otimes t^{r+s}+
(\bar{r},s+\beta)(v_1\otimes t^{r+s})$
\end{center}
for all $v_1 \in V^\prime ,x \in \mathfrak{g}(\underline{k},0)$ and $D(\bar{r},r)\in H_{n-2}(m')$. \\
For $v\in M$, let $\bar{v}$ be the image of $v$ in $\widetilde{V}$. Now define 
\begin{center}
$\phi: M \rightarrow L(\widetilde{V})$
\end{center}
\begin{center}
by $v\mapsto \bar{v}\otimes t^k$ for $v\in M(k).$
\end{center}
This map is clearly a nonzero $\tau^0$ -module homomorphism. Hence by irreducibility
of $M$ it follows that $M\cong \phi (M)$ is a $\tau ^0$ submodule of
$L(\widetilde {V})$. \\Clearly $L$ is naturally $\Lambda$ graded. Now since $M$
and $W$ are $Z^{n-2}$ graded, therefore they are naturally $\Lambda$ graded and hence so is
$\widetilde{V}$. Therefore $\widetilde{V}=\oplus_{\bar{p}\in
\Lambda}\widetilde{V}(\bar{p})$.\\
Now for $\bar{p} \in \Lambda$, we set 
\begin{center}
$L(\widetilde{V})(\bar{p})=\{v\otimes t^{k+r+p}|v\in \widetilde{V}(\bar{k}),r \in
\Gamma,k\in \mathbb{Z}^{n-2}\}$
\end{center}
It can be easily verified that $L(\widetilde{V})(\bar{p})$ is a $\tau^0$ submodule
of $L(\widetilde{V})$.\\
Let $I(\bar r, r)=D(\bar r,r)-D(\bar r, 0)$ for all $r \in \Gamma$. It is easy to observe that $H_{n-2}^\prime(m')$ is a subalgebra of $H_{n-2}(m')$.
The following result can be deduced similarly as in \cite{[4]}.  
\begin{proposition}\label{prop 6.1}
\itemize
\item[(1)] $M\cong L(\widetilde{V})(\bar{0})$ as $\tau^0$ -module.
\item[(2)] $\widetilde{V}$ is $\Lambda$ -graded-irreducible module over $L$.
\item[(3)] $\widetilde{V}$ is completely reducible module over $L$ and all its
irreducible components are mutually isomorphic as $H_{n-2}^{\prime}(m')\ltimes
\mathfrak{h}(\underline{0})\otimes A_{n-2}(m')$ -modules.
\end{proposition}
Now we will concentrate on the irreducible representation of $L$. Let $(W,\pi)$ be a
finite-dimensional representation of $L$. Let $\pi (L(\mathfrak{g}^{\prime},\sigma))=\mathfrak{g}^1$, then
$\pi(L)=\mathfrak{g}^1\oplus \mathfrak{g}^2$, where $\mathfrak{g}^2$ is the unique complement of $\mathfrak{g}^1$ in $\mathfrak{gl}(W)$
(Proposition 19.1(b) of \cite{[8]} ). So $W$ will be an irreducible module for
$\mathfrak{g}^1\oplus \mathfrak{g}^2$. Therefore $W\cong W_1\otimes W_2$, where $W_1$ and $W_2$ are
irreducible modules for $\mathfrak{g}^1$ and $\mathfrak{g}^2$ respectively (see \cite{[9]} ). Let
$\mathfrak{g}^{\prime}=\mathfrak{g}^{\prime}_{ss}\oplus R$, where $\mathfrak{g}^{\prime}$ and $R$ are Levi and
radical part of $\mathfrak{g}^\prime$. Then as $\sigma_i(\mathfrak{g}^\prime)=\mathfrak{g}^\prime$ and
$\sigma_i(R)=R$ for $1\leq i\leq n$, we have
$L(\mathfrak{g}^\prime,\sigma)=L(\mathfrak{g}^\prime_{ss},\sigma)\oplus L(R,\sigma)$. Now $W_1$ is
irreducible module for $L(\mathfrak{g}^\prime,\sigma)$. As $R$ is a solvable ideal, it follows
that $\pi(L(R,\sigma))$ lies in the center of $\pi(L)$, which is at most one
dimensional. Hence $L(R,\sigma)$ acts as a scalar on $W$. So $W_1$ will be an
irreducible module for $L(\mathfrak{g}^\prime_{ss},\sigma)$. \\

Fix a positive integer $l$. For each $i$, let $a_i=(a_{i,1},\dots ,a_{i,l})$ such that
\begin{align}\label{a6.1}
a_{i,j}^{m_i}\neq a_{i,t}^{m_i}, \, for \, j\neq t.
\end{align}
Now we recall a theorem from \cite{[10]}. Let $\mathfrak{G}$ be a finite-dimensional
semisimple Lie algebra. Let $\sigma_1, \dots \sigma_n$ be finite order automorphisms
on $g$ of order $m_1,\dots m_n$ respectively. Let $L(\mathfrak{G},\sigma)$ be the corresponding
multiloop algebra. Let $I=\{(i_1,i_2,\dots, i_n)|1\leq i_j\leq l\}$. Now for
$S=(i_1,i_2,\dots ,i_n)\in I$ and $r=(r_1,r_2,\dots r_n)\in \mathbb{Z}^n$,
$a_S^r=a_{1,i_1}^{r_1}a_{2,i_2}^{r_2}\cdots a_{n,i_n}^{r_n}.$ Now consider the
evaluation map $\phi : \mathfrak{G}\otimes A\rightarrow \bigoplus \mathfrak{G}$ ($l^n$ copies), $\phi
(X\otimes t^r)=(a_{I_1}^rX,a_{I_2}^rX, \dots,a_{I_{l^n}}^rX)$, where $I_1,I_2, \cdots
I_{l^n}$ are the elements of $I$. Now consider the restriction of $\phi$ to
$L(\mathfrak{G},\sigma)$.
\begin{theorem}$($\cite{[10]}$)$\label{Th 6.1}
Let $W^\prime$ be a finite-dimensional irreducible representation of $L(\mathfrak{G},\sigma)$.
Then the representation factors through $\bigoplus \mathfrak{G}$ $($ $l^n$ copies$)$.
\end{theorem} 
In our case $W_1$ is irreducible module for $L(\mathfrak{g}^\prime _{ss},\sigma)$, so the representation will factors through $l^{n-2}$ copies of $\mathfrak{g}^{\prime}_{ss}$.
\begin{proposition}
Let $W_1$ be irreducible module for $L(\mathfrak{g}^\prime _{ss},\sigma)$ as above. Then the
representation of $L(\mathfrak{g}^\prime _{ss},\sigma)$ factors through only one copy of
$\bigoplus \mathfrak{g}^\prime _{ss}$. So $\mathfrak{g}^1_{ss}\cong \mathfrak{g}^\prime_{ss}$.
\end{proposition}
\begin{proof}
We know by Theorem \ref{Th 6.1}, that the representation factors through $l^{n-2}$
copies, for some positive integer $l$. we will prove here that $l=1$. Choose the $i$-th piece of $\mathfrak{g}^\prime_{ss}$ and choose the
projection of the map $\pi$, say $\pi_i$ onto it. Doing the same calculation as in
\cite{[3]} we will get $\pi_i(H_{n-2}^{\prime}(m'))=0$ and $a_{I_i}^r=1$ for all $r\in
\Gamma$. Now suppose there are at least two pieces, say $i$-th and $j$-th piece is
there. Therefore $I_i$ and $I_j$ are two different element of $I$ with
$a_{I_i}^r=1=a_{I_j}^r$ for all $r \in \Gamma$. Let $I_i=(i_1,i_2,\dots, i_{n-2})$ and
$I_j=(j_1,j_2,\dots j_{n-2})$. Therefore there is $k$ with $1\leq k\leq n-2$ such that
$i_k\neq j_k$. Now if we take $r=(0,\dots, m_k,\dots, 0)$ then
$a_{I_i}^r=1=a_{I_j}^r$ will give $a_{k,i_k}^{m_k}=a_{k,j_k}^{m_k}$, a contradiction
to equation (\ref{a6.1}). So there is at most one piece.
\end{proof}
Now we know $\pi_i(H_{n-2}^\prime(m'))=0$, therefore $\mathfrak{g}^2\subseteq \pi(H_{n-2}^\prime(m'))$.
Now our aim is to understand  finite dimensional irreducible modules for
$H_{n-2}^\prime(m)$. We are going to establish a relation between finite-dimensional $H_{n}^\prime(m)$ modules
and $H_{n}(m)\ltimes A_{n}(m)$-modules with finite-dimensional weight spaces. In order to do that we follow the method used in \cite{[3]}.

Let $W$ be an irreducible finite-dimensional $H_n^\prime(m)$ module.  We define  $A_n(m)\rtimes H_n(m)$ module action on $L(W)=W\otimes A_n(m)$ in the following way
\newline

$D(\bar{r},r).(w\otimes t^k)=(I(\bar{r},r).w)\otimes t^{r+k}+(\bar{r}, \beta+k)w\otimes t^{r+k}$, \\
\hspace*{4mm}$D(u,0)(w\otimes t^k)=(u,\alpha+k)w\otimes t^k$,\\
\hspace*{4mm}
$ t^r.w \otimes t^k=w \otimes t^{r+k}$, for all $u\in \mathbb{C}^n,\;  r(\neq0),k\in \Gamma, w\in W$ and some $\alpha,\beta\in \mathbb{C}^n$.

 We denote this module as $(\pi_{\alpha,\beta},L(W))$. \\
It is clear that $L(W)=\displaystyle{\bigoplus_{k\in \Gamma }}W\otimes t^k$ is the weight space decomposition of $L(W)$. Consider the $H_n^\prime(m)$ submodule of $L(W)$, $W_0=\{w\otimes t^r-w\;|\; r\in \Gamma,\;w\in W\}$, then $\overline{L(W)}=L(W)/W_0$ is an $H_n^\prime(m)$ submodule. Let $(\theta,\overline{L(W)})$ be its corresponding representation. Now we define a new representation $(\theta_\xi,\overline{L(W)})$ of $H_n^\prime(m)$ by the following action
\begin{center}
$\theta_{\xi}(I(\bar{r},r)).v=\theta (I(\bar{r},r))v+(\bar{r},\xi)v$, where $\xi \in \mathbb{C}^n$. 
\end{center}
 It is easy to see $\theta_0=\theta$. 
 \begin{lemma}\label{l6.2}
 Let $W$ be a finite-dimensional irreducible $H_n^\prime(m)$ module. Then $L(W)$ is irreducible $A_n(m) \rtimes H_n(m)$ module.
 \end{lemma}
 \begin{proof}
 Let $U_0$ be a non-zero submodule of $L(W)$. Since $L(W)$ is a weight module , so $u_0 \otimes t^r \in U_0$ for some non-zero $u_0 \in U_0$ and $r \in \Gamma$. Now the action of $A_n(m)$ on $L(W)$ implies that $u_0\otimes A_n(m) \in U_0.$ Hence we have $U_0= U_1\otimes A_n(m)$ for some non-zero subspace $U_1$ of $W$. Therefore to complete the proof, it is sufficient to show that $U_1$ is a submodule of $W$.\\
 Let $r \in \Gamma$, $u \in U_1$ and consider the action
 $$ D(\bar r,r)(u \otimes t^{-r})= I(\bar r,r).u + (\bar r,\beta -r)u.$$
 This implies that $I(\bar r,r).u \in U_1$. Hence $U_1$ is a submodule of $W$.
 \end{proof}
\begin{lemma}\label{l6.3}
Let $(\pi_{\alpha,\beta},L(W))$ be a finite dimensional irreducible $A_n(m)\rtimes H_n(m)$ module for a finite dimensional module $(\eta,W)$ of $H_n ^\prime(m)$. Then $(\theta_\xi, \overline{L(W)}$ is irreducible for $H_n^\prime(m)$.
\end{lemma} 
\begin{proof}
Note that $(\theta_{\alpha-\beta},\overline{L(W)}) \cong(\eta,W)$. Moreover one can see that if $W_0$ is a nonzero proper submodule of $W$, then $W_0\otimes A_n(m)$ is a nonzero proper submodule of $L(W)$. 
\end{proof} 
 \begin{lemma}\label{l4.4}

Let $W$ be finite dimensional irreducible module for $\mathfrak{sp}_n.$  Then $W$ can be made into
$H_n^\prime(m)$-module by the action: $I(\bar r,r).w=(r^t {\bar r})w+(\bar r,\zeta)w$,
where $\zeta \in \mathbb{C}^n$ and $r^t $ denote the transpose of the row vector $r\in \Gamma.$

\end{lemma}
 \begin{proof}
 Note that for $r \in \Gamma ,$ $r^t\bar r \in \mathfrak{sp}_n$ and hence the action is well defined. It is easy to see that this action defines a module structure on $W$.

  \end{proof}
\begin{theorem}
Suppose $W$ is a finite-dimensional irreducible $\mathfrak{sp}_n$-module.  Let $\alpha,\beta \in \mathbb{C}^n$. Take
$L(W)=W\otimes A_n(m)$ and consider the action $D(\bar r,r)(w\otimes t^k)=(\bar r,k+\beta)w\otimes
t^{k+r}+(r^t{\bar r}).w\otimes t^{k+r}$, for $r(\neq 0) \in \Gamma$ and $D(u,0)(w\otimes
t^k)=(u,\alpha+k)w\otimes t^k$ and $t^r(w\otimes t^k)=w\otimes t^{k+r}$. Then $L(W)$ is an irreducible module for $H_n(m) \ltimes   A_n(m)$. Moreover, all
irreducible representations of $H_n(m)\ltimes A_n(m)$ with finite-dimensional weight spaces
occur in this way.
\end{theorem}
\begin{proof}
 The proof follows from Lemmas \ref{l6.2}, \ref{l6.3},\ref{l4.4} and Theorem 5.2 of \cite{[11]}.
\end{proof}
We know $\widetilde{V}$ is completely reducible $L$ module. Therefore
$\widetilde{V}=\oplus_{i=1}^K \widetilde{V}_i$ for some $K\in \mathbb{N}$. Then by
the previous discussion each $\widetilde{V}_i\cong W_1^{i}\otimes W_2^i$ as
$\mathfrak{g}^\prime_{ss}\oplus \mathfrak{ sp}_{n-2}$ module, where $W_1^i$, $W_2^i$ are irreducible modules
for $\mathfrak{g}^\prime_{ss}$ and $\mathfrak{ sp}_{n-2}$ respectively. Since each component $\widetilde{V}_i$
is isomorphic as $H_{n-2}^\prime  (m')\ltimes (\mathfrak{h}(\underline{0})\otimes A(m))$ module, we can
take $W_2^i\cong W_2^1$ ($ =W_2$ ,say) as $\mathfrak{ sp}_n$-modules for each $i \in \{1,\dots
,K\}$. Now consider $W_1=\sum_{i=1}^KW_1^i$, which is a
$L(\mathfrak{g}^\prime_{ss},\sigma)$-module,  in particular $\mathfrak{g}^\prime_{ss}$ module. Since each
$W_1^i$ is irreducible, without loss of generality we can assume that the above sum
is direct. It is easy to see that $L$ is $\Lambda$-graded with zeroth component
$H_{n-2}^\prime  (m')\ltimes (\mathfrak{h}(\underline{0})\otimes A(m))$ and since $\widetilde{V}$ is
$\Lambda$ graded irreducible module (Proposition \ref {prop 6.1}), we can take $W_1$ as
$\Lambda$-graded irreducible $L(\mathfrak{g}^\prime_{ss},\sigma)$ module and $W_2$ a zero
graded as $H_{n-2}^\prime  (m')$-module which lies inside the zeroth graded component of $L$.\\
We now define a $\tau^0$-module structure on $W_1\otimes W_2\otimes A_n$ by 
\begin{center}
$X\otimes t^k(w_1\otimes w_2\otimes t^l)=Xw_1\otimes w_2\otimes t^{k+l}$, for $k,l
\in \mathbb{Z}^{n-2}$ and $X\in \mathfrak{g}^\prime _{ss}(\underline{k})$.
\end{center}
\begin{center}
$D(\bar r,r)(w_1\otimes w_2\otimes t^k)=(u,k+\beta)w_1\otimes w_2\otimes
t^{k+r}+w_1\otimes( r {\bar r}^t).w_2\otimes t^{k+r}$ for $r\neq 0$,

\end{center}
\begin{center}
$D(u,0)(w_1\otimes w_2 \otimes t^k)=(u,k+\alpha)w_1\otimes w_2\otimes t^k$.
\end{center}
Now take any one-dimensional representation of $L(R,\sigma)$ say $\psi$. Then for
$y\in R(\underline{k})$ we take $y\otimes t^k(w_1\otimes w_2\otimes
t^l)=\psi(y)(w_1\otimes w_2 \otimes t^{k+l})$.  Since $W_1$ is $\Lambda$ graded,
which is compatible with $\Lambda$-gradation of $\mathfrak{g}^\prime_{ss}$, so the submodule
$V^\prime=\bigoplus_{k \in \mathbb{Z}^n} W_{1,k}\otimes W_2\otimes t^k$ will be
irreducible module for $\tau^0$. One can easily check that
$L(\widetilde{V})(\bar{0})\cong V^\prime$ as $\tau^0$-module. Now from Proposition \ref{prop 6.1}, we see that $M\cong V^\prime$. \\

\begin{theorem}\label{Thm 6.6}
Let $V$ be an irreducible integrable $\tau$ module with finite-dimensional weight spaces, with $K_k$ acting as $c_0$ and $K_i$ acts trivially for $1\leq i\leq n$ and $i\neq k$. Then $V\cong U(\tilde{\tau})M/M^{\textit{Rad}}$, where $M^{Rad}$ is the unique maximal submodule of $U(\tilde \tau)M$. 
\end{theorem}
\begin{proof}
Follows from the previous discussion.
\end{proof}

 
 	\vspace{2cm}
		{\bf Acknowledgments:} Authors would like to thank the anonymous referee for nice and helpful suggestions. For this project the second author is supported by funds of the National Natural Science Foundation of China (Grant No. 12071136) and Science and Technology Commission of Shanghai Municipality (Grant No. 22DZ2229014).

\end{document}